\documentclass[12pt,reqno]{amsart}
\usepackage{amsmath,amssymb,amscd,latexsym,amsthm,mathrsfs}
\usepackage{color}
\usepackage[unicode]{hyperref}
\usepackage{hypbmsec,longtable}
\ifx\pdfoutput\undefined\else
\hypersetup{
colorlinks=true,
linkcolor=blue,
citecolor=blue,
urlcolor=blue,
filecolor=blue,
bookmarksnumbered=true,
pdfstartview=FitH,
pdfhighlight=/N
}
\fi
\newtheorem{theorem}{Theorem}
\newtheorem{lemma}{Lemma}
\newtheorem{corollary}{Corollary}

\theoremstyle{remark}

\newtheorem*{acknowledgements}{Acknowledgements}
\renewcommand{\d}{{\mathrm d}}

\begin{document}

\hypersetup{pdfauthor={Mathew Rogers},%
pdftitle={Ramanujan zeta}}

\title{Identities for the Ramanujan zeta function}

\author{Mathew Rogers}
\address{Department of Mathematics and Statistics, Universit\'e de Montr\'eal,
CP 6128 succ.\ Centre-ville, Montr\'eal Qu\'ebec H3C\,3J7, Canada}
\email{mathewrogers@gmail.com}

\date{March 28, 2013.}


\begin{abstract}
We prove formulas for special values of the Ramanujan tau zeta function.  Our formulas show that $L(\Delta, k)$ is a
period in the sense of Kontsevich and Zagier when $k\ge12$.  As an illustration, we reduce $L(\Delta, k)$ to explicit integrals of hypergeometric and algebraic functions when $k\in\{12,13,14,15\}$.
\end{abstract}

\maketitle
\section{Introduction}

\subsection{Background and previous results}
\label{s-intro}
Ramanujan introduced his zeta function in 1916 \cite{Ra}. Following Ramanujan, let
\begin{equation*}
\Delta(z):=q\prod_{n=1}^{\infty}(1-q^n)^{24}=\sum_{n=1}^{\infty}\tau(n) q^n,
\end{equation*}
where $q=e^{2\pi i z}$.  Ramanujan observed that $\tau(n)$ is multiplicative,
and this lead him to study the Dirichlet series:
\begin{equation*}
L(\Delta,s):=\sum_{n=1}^{\infty}\frac{\tau(n)}{n^s}.
\end{equation*}
Mordell subsequently proved that $L(\Delta,s)$ has an Euler product, and satisfies the functional equation $(2\pi)^{s-12}\Gamma(12-s)L(\Delta,12-s)=(2\pi)^{-s}\Gamma(s)L(\Delta,s)$ \cite{M}, \cite[pg.~242]{ZA}.  The main goal of this paper is to prove formulas for $L(\Delta,k)$ when $k$ is a positive integer.

Kontsevich and Zagier defined a \emph{period} to be a number which can be expressed as a multiple integral of algebraic functions, over a domain described by algebraic equations \cite{KZ}.  The ring of periods contains both the algebraic numbers, and certain transcendental numbers like $\pi$ and $\log 2$.  It follows from the work of Beilinson \cite{Be}, and Deninger and Scholl \cite{DS}, that special values of $L$-functions attached to modular forms are also periods.  Paraphrasing \cite[p.~24]{KZ}, their results follow from deep cohomological manipulations, and a careful study of values of regulators.

Following Deligne, we say that $L(\Delta,k)$ is a critical $L$-value if $1\le k\le 11$.  Kontsevich and Zagier summarized the properties of critical $L$-values in \cite{KZ}.  It is relatively easy to relate these values to integrals of algebraic functions.  The standard Mellin transform gives
\begin{equation}\label{critical q-integral}
L(\Delta,k)=\frac{(2\pi)^{k}}{(k-1)!}\int_{0}^{\infty}u^{k-1} \Delta(i u) \d u
\end{equation}
whenever $k\ge 1$.  The usual method for obtaining an elementary integral, is to set
\begin{equation}\label{variable change}
u=\frac{F(1-\alpha)}{2 F(\alpha)},
\end{equation}
where $F(\alpha)$ is the classical hypergeometric series:
\begin{align}
F(\alpha)&:=\sum_{n=0}^{\infty}{2n\choose n}^2\left(\frac{\alpha}{16}\right)^n\label{2F1 def}\\
&=\frac{2}{\pi}\int_{0}^{1}\frac{\d u}{\sqrt{(1-u^2)\left(1-\alpha u^2\right)}}.\label{elliptical integral}
\end{align}
Notice that $2i u$ is the period ratio of the elliptic curve $y^2=(1-x^2)(1-\alpha x^2)$.  Now appeal to the classical formulas \cite[pg.~124]{Be3}, \cite[pg.~120]{Be3}:
\begin{align}
\Delta(i u)&=\frac{1}{16}\alpha(1-\alpha)^{4}\left[F(\alpha)\right]^{12},\\
\frac{\d u}{\d\alpha}&=\frac{-1}{2\pi\alpha(1-\alpha)\left[F(\alpha)\right]^2},\label{duda}
\end{align}
and notice that $\alpha\in(1,0)$ when $u\in(0,\infty)$.  Equation \eqref{critical q-integral} reduces to
\begin{equation}\label{Critical L-value formula}
L(\Delta,k)=\frac{\pi^{k-1}}{16(k-1)!}\int_{0}^{1}(1-\alpha)^3 \left[F(\alpha)\right]^{11-k}\left[F(1-\alpha)\right]^{k-1} \d\alpha.
\end{equation}
Substituting \eqref{elliptical integral} allows us to deduce that $\pi^{11-k}L(\Delta,k)$ is a period if $1\le k\le 11$.  Kontsevich and Zagier illustrated this point with an equivalent identity \cite[pg.~24]{KZ}.

\subsection{Main results}
    It is not at all obvious that $L(\Delta,k)$ is a period if $k\ge 12$.  The method from the previous section fails, because the integrand in equation \eqref{Critical L-value formula} becomes a ratio of algebraic and hypergeometric functions.  It seems to be very difficult to pass from equation \eqref{Critical L-value formula} to anything interesting.  In this paper we will use ideas from the philosophy established jointly with Zudilin \cite{Rg1}, \cite{Rg2}, \cite{ZP}, to reduce these $L$-values to integrals of algebraic and hypergeometric functions.   For example, we prove the following theorem:
\begin{theorem}\label{Theorem value at 12} The following identity is true:
\begin{equation}\label{intro formula}
\begin{split}
L(\Delta,12)=-\frac{128\pi^{11}}{8241\cdot11!} \int_{0}^{1}& \left[F(\alpha)F(1-\alpha)\right]^5\\
 &\times\left(\frac{2+251 \alpha+876 \alpha^2+251 \alpha^3+2\alpha^4}{1-\alpha}\right)~\log\alpha~\d\alpha,
\end{split}
\end{equation}
where $F(\alpha)$ is defined in \eqref{2F1 def}.
\end{theorem}
Formula \eqref{intro formula} shows that $L(\Delta,12)$ belongs to the ring of periods.  Most of our identities involve hypergeometric functions, but these always reduce to elliptic integrals by \eqref{elliptical integral}.  Equation \eqref{intro formula} becomes a massive twelve dimensional integral:
\begin{equation}\label{12d integral}
\begin{split}
L(\Delta,12)=\frac{512 \pi}{1284977925}\int_{0}^{1}& \left[\int_{0}^{1}\int_{0}^{1}\frac{\d u~\d z}{\sqrt{(1-u^2)(1-z^2)(1-\alpha u^2)(1-(1-\alpha)z^2)}}\right]^5\\
 &\times\left(\frac{2+251 \alpha+876 \alpha^2+251 \alpha^3+2\alpha^4}{1-\alpha}\right)~\int_{\alpha}^{1}\frac{1}{t}\d t~\d\alpha.
\end{split}
\end{equation}
Theorem \ref{main theorem} provides the key formulas we need to express $L(\Delta,k)$ in terms of two dimensional integrals of algebraic and hypergeometric functions for $k\ge 12$.  Corollary \ref{Values at k bigger than 12} highlights examples when $k\in\{13,14,15\}$.  The $k=12$ case is apparently the only instance where a reduction to a one-dimensional integral is possible, and we discuss this case separately in Section \ref{Sec:12 case}.  Finally, we note that equation \eqref{intro formula} is closely related to the \textit{moments of elliptic integrals} studied in \cite{BBGW}, \cite{W}, and \cite{Z}.

\section{A formula for $L(\Delta,12)$}\label{Sec:12 case}
  The main goal of this section is to prove a formula for $L(\Delta,12)$ using the method developed in \cite{Rg1} and \cite{Rg2}.  In Section \ref{higher cases} we study $L(\Delta,k)$ for arbitrary $k\ge 12$.

The crucial first step is to decompose $\Delta(z)$ into a linear combination of products of two Eisenstein series.  The usual Eisenstein series is defined by
\begin{align*}
E_{k}(z):=1+\frac{2}{\zeta(1-k)}\sum_{n=1}^{\infty}\frac{n^{k-1}e^{2\pi i n z}}{1-e^{2\pi i n z}}.
\end{align*}
The most famous decomposition of $\Delta(z)$ is due to Ramanujan:
\begin{equation}\label{Ramanujan's decomposition}
1728\Delta(z)=E_4(z)^3-E_{6}(z)^2,
\end{equation}
but this formula involves $E_4(z)^3$, and the method from \cite{Rg1} and \cite{Rg2} only applies to modular forms which decompose into products of two Eisenstein series. We avoid this obstruction by considering a linear combination of $\Delta$'s.
\begin{lemma}\label{decomp lemma}  We have
\begin{equation}\label{decomposition}
\begin{split}
\Delta(z)+24\Delta(2z)+2^{11}\Delta(4z)=\frac{8}{504^2}&\left[E_6(2z)-64E_6(4z)\right]\\
&\times\left[E_6(z)-33E_{6}(2z)+32E_6(4z)\right].
\end{split}
\end{equation}
\end{lemma}
\begin{proof} It is easy to show that both sides of the equation are modular forms on $\Gamma_{0}(4)$.  By the standard valence formula for congruence subgroups \cite{Sc}, the two sides are equal if their Fourier series expansions agree to more than $[\Gamma(1):\Gamma_{0}(4)]=6$ terms.  We used a computer to check that the first $1000$ Fourier coefficients agree, and thus we conclude that the identity is true.  It is possible to construct an alternative elementary proof by applying formulas from Ramanujan's notebooks (use \cite[pg.~124, Entry~12]{Be3} and \cite[pg.~126, Entry~13]{Be3}).
\end{proof}
Now apply an involution to the first term on the right-hand side of \eqref{decomposition}.  The equation becomes
\begin{equation*}
\begin{split}
\Delta(z)+24\Delta(2z)+2^{11}\Delta(4z)=\frac{8}{504^2(2z)^6}&\left[E_6\left(\frac{-1}{2z}\right)-E_6\left(\frac{-1}{4z}\right)\right]\\
&\times\left[E_6(z)-33E_{6}(2z)+32E_6(4z)\right].
\end{split}.
\end{equation*}
Suppose that $z=i u$ with $u\ge 0$.  The right-hand side reduces to a four-dimensional infinite series.  We have
\begin{equation*}
\Delta(i u)+24~\Delta(2i u)+2^{11}~\Delta(4i u)=\frac{1}{8 u^6}\sum _{\substack{n,m,r,s\ge 1\\\text{$m$,$r$,$s$ odd}}} (n r)^5 e^{-2\pi\left(\frac{  n m}{4 u}+r s u\right)}.
\end{equation*}
Multiply both sides by $u^{k-1}$ and integrate for $u\in(0,\infty)$.  By uniform convergence:
\begin{equation}\label{almost done 1}
\begin{split}
8\left(1+24\cdot 2^{-k}+2^{11-2k}\right)L(\Delta,k)=\frac{(2\pi)^{k}}{(k-1)!}\sum_{\substack{n,m,r,s\ge 1\\\text{$m$,$r$,$s$ odd}}}(n r)^5 &\int_{0}^{\infty}u^{k-7}e^{-2\pi\left(\frac{n m}{4u}+r s u\right)}\d u.
\end{split}
\end{equation}
The rational term on the left cancels the Euler factor of $L(\Delta,s)$ at the prime $p=2$.  Notice that
\begin{equation*}
\left(1+24\cdot2^{-k}+2^{11-2k}\right)L(\Delta,k)=\sum_{\substack{n=1\\\text{$n$ odd}}}^{\infty}\frac{\tau(n)}{n^s}.
\end{equation*}
Finally apply the key trick: \textit{Use a change of variables to swap the indices of summation inside the integral in \eqref{almost done 1}.  If the trick is properly executed, then the right-hand side reduces to an integral involving modular functions.}

\begin{proof} [Proof of Theorem \ref{Theorem value at 12}] Set $k=12$ and let $u\mapsto n u/r$ in \eqref{almost done 1}.  The formula becomes
\begin{align*}
\frac{8241}{2^{22}\pi^{12}}L(\Delta,12)=&\frac{1}{11!}\int_{0}^{\infty}u^{5}\sum_{\substack{r,m\ge 1\\\text{$r$,$m$ odd}}}\frac{1}{r}e^{-\frac{2\pi r m}{4u}}\sum_{\substack{n,s\ge 1\\\text{$s$ odd}}}n^{11} e^{-2\pi n s u}  \d u\\
=&\frac{1}{11!}\int_{0}^{\infty}u^{5}\log\left(\prod_{\substack{m=1\\\text{$m$ odd}}}^{\infty}\frac{1+e^{-\frac{2\pi m}{4u}}}{1-e^{-\frac{2\pi m}{4u}}}\right)\sum_{n=1}^{\infty}\frac{n^{11} e^{-2\pi n u}}{1-e^{-4\pi n u}}  \d u.
\end{align*}
Now suppose that $u$ and $\alpha$ are related by \eqref{variable change}.  In particular:
\begin{equation*}
u=\frac{F(1-\alpha)}{2 F(\alpha)}.
\end{equation*}
Then $\alpha\in(1,0)$ when $u\in(0,\infty)$.  We compute $\d u/\d \alpha$ using \eqref{duda}, and various identities from Ramanujan's notebooks imply
\begin{align}
\prod_{\substack{m=1\\\text{$m$ odd}}}^{\infty}\frac{1+e^{-\frac{2\pi m}{4u}}}{1-e^{-\frac{2\pi m}{4u}}}&=\alpha^{-1/8},\label{L12 extra 1}\\
\sum_{n=1}^{\infty}\frac{n^{11} e^{-2\pi n u}}{1-e^{-4\pi n u}}=&\frac{1}{32}\alpha\left(2+251 \alpha+876 \alpha^2+251 \alpha^3+2 \alpha^4\right) \left[F(\alpha)\right]^{12}.\label{L12 extra 2}
\end{align}
We can prove \eqref{L12 extra 1} using \cite[pg.~124, Entry~12]{Be3}, and the proof of \eqref{L12 extra 2} follows from the Eisenstein series relation $E_{12}(z)=E_{6}(z)^2$ and \cite[pg.~126, Entry~13]{Be3}.
Combining the various formulas completes the proof.
\end{proof}

\section{Formulas for $L(\Delta,k)$ when $k\ge 12$}\label{higher cases}

In this section we obtain double integrals for $L(\Delta,k)$ when $k\ge 12$.  In general, it seems to be difficult to simplify formulas for $L(f,k)$, when $k>\text{weight}(f)$.  So far there is only one instance where such an $L$-value has been reduced to recognizable special functions.  Zudilin proved that
\begin{equation}\label{L(g,3)}
\begin{split}
\frac{768\sqrt{2}}{\pi^{3/2}}L(g,3)=&\Gamma^2\left(1/4\right){_4F_3}\left(\substack{1,1,1,\frac12\\\frac74,\frac32,\frac32};1\right)
+24\Gamma^2\left(3/4\right){_4F_3}\left(\substack{1,1,1,\frac12\\\frac54,\frac32,\frac32};1\right)\\
&+3\Gamma^2\left(1/4\right){_4F_3}\left(\substack{1,1,1,\frac12\\\frac34,\frac32,\frac32};1\right).
\end{split}
\end{equation}
where $g(z)=\eta^2(4z)\eta^2(8z)$ is the weight $2$ CM newform attached to conductor $32$ elliptic curves \cite{ZP}, \cite{Zt}.  Rodriguez-Villegas and Boyd used numerical experiments to find many relations between Mahler measures and values of $L(f,k)$, but all of their conjectures are still open \cite{Bo}, \cite{RVTV}.

\begin{theorem}\label{main theorem}Assume that $k\ge 12$.  If $k$ is odd:
\begin{equation}\label{odd case}
\begin{split}
8&\left(1+24\cdot 2^{-k}+2^{11-2k}\right)L(\Delta,k)\\
&\qquad=\frac{2^{k-7} i^{k-1}\pi^{2k-11}\left[\zeta(6-k)\right]^2}{(k-1)!(k-12)!}\int_{0}^{\infty}\int_{u}^{\infty}(z-u)^{k-12} P_{k-5}(u)Q_{k-5}(z)\d z~\d u,
\end{split}
\end{equation}
where
\begin{align}
P_k(u):=&E_{k}(i u)-(2+2^{k})E_{k}(2i u)+2^{k+1}E_{k}(4 i u),\label{P def}\\
Q_k(z):=&E_{k}(i z)-E_{k}(2 i z)\label{Q def}.
\end{align}
If $k$ is even:
\begin{equation}\label{even case}
\begin{split}
8&\left(1+24\cdot 2^{-k}+2^{11-2k}\right)L(\Delta,k)\\
&\qquad=-\frac{2^{k-1}i^{k} \pi^{2k-11}\zeta(1-k)\zeta(11-k)}{(k-1)!(k-12)!}\int_{0}^{\infty}\int_{u}^{\infty}(z-u)^{k-12} u^5 R_k(u)S_{k-10}(z)\d z~ \d u,
\end{split}
\end{equation}
where
\begin{align}
R_{k}(u):=&E_{k}(2i u)-2^{k}E_{k}(4 i u),\label{R def}\\
S_{k}(z):=&E_{k}(i z)-(1+2^{k-1})E_{k}(2i z)+2^{k-1}E_{k}(4i z).\label{S def}
\end{align}
\end{theorem}
\begin{proof}
We prove \eqref{odd case} first.  Assume that $k$ is odd in \eqref{almost done 1}, and let $u\mapsto m u/r$.  The integral becomes
\begin{equation*}
\begin{split}
8&\left(1+24\cdot 2^{-k}+2^{11-2k}\right)L(\Delta,k)\\
&\qquad=\frac{(2\pi)^{k}}{(k-1)!}\int_{0}^{\infty}u^{k-7}\sum_{\substack{m\ge 1\\\text{$m$ odd}}}\frac{m^{k-6} e^{-2\pi m u}}{1-e^{-4\pi m u}}\sum_{\substack{n,r\ge 1\\\text{$r$ odd}}}\frac{n^5}{r^{k-11}}e^{-\frac{2\pi n r}{4u}}\d u.
\end{split}
\end{equation*}
Let $u\mapsto\frac{1}{4u}$, then
\begin{equation}\label{gen odd case intermediate}
\begin{split}
8&\left(1+24\cdot 2^{-k}+2^{11-2k}\right)L(\Delta,k)\\
&\qquad=\frac{(2\pi)^{k}}{(k-1)!}\int_{0}^{\infty}4(4u)^{5-k}\sum_{\substack{m\ge 1\\\text{$m$ odd}}}\frac{m^{k-6} e^{-\frac{2\pi m}{4u}}}{1-e^{-\frac{4\pi m}{4u}}}\sum_{\substack{n,r\ge 1\\\text{$r$ odd}}}\frac{n^5}{r^{k-11}}e^{-2\pi n r u}\d u.
\end{split}
\end{equation}
Using the involution for $E_{k}(u)$, we easily find that
\begin{equation}\label{gen odd case nested sum 1}
\begin{split}
4(4u)^{5-k}\sum_{\substack{m\ge 1\\\text{$m$ odd}}}\frac{m^{k-6} e^{-\frac{2\pi m}{4u}}}{1-e^{-\frac{4\pi m}{4u}}}=i^{k-1}2^{5-k}\zeta(6-k)P_{k-5}(u),
\end{split}
\end{equation}
where $P_{k}(u)$ is defined in \eqref{P def}.
To simplify the second sum in \eqref{gen odd case intermediate}, we require the integral:
\begin{equation}\label{shifted gamma integral}
\frac{e^{-2\pi a u}}{a^s}=\frac{(2\pi)^s}{\Gamma(s)}\int_{u}^{\infty}(z-u)^{s-1}e^{-2\pi a z}\d z.
\end{equation}
If $s\mapsto k-11$ and $a\mapsto n r$, then
\begin{align}
\sum_{\substack{n,r\ge 1\\\text{$r$ odd}}}\frac{n^5}{r^{k-11}}e^{-2\pi n r u}=&\sum_{\substack{n,r\ge 1\\\text{$r$ odd}}}\frac{n^{k-6}}{(n r)^{k-11}}e^{-2\pi n r u}\notag\\
=&\frac{(2\pi)^{k-11}}{(k-12)!}\int_{u}^{\infty}(z-u)^{k-12}\sum_{\substack{n,r\ge 1\\\text{$r$ odd}}}n^{k-6}e^{-2\pi n r z}\d z\notag\\
=&\frac{(2\pi)^{k-11}\zeta(6-k)}{2(k-12)!}\int_{u}^{\infty}(z-u)^{k-12}Q_{k-5}(z)\d z,\label{gen odd case nested sum 2}
\end{align}
where $Q_{k}(z)$ is defined in \eqref{Q def}.  Finally combine \eqref{gen odd case nested sum 2}, \eqref{gen odd case nested sum 1}, and \eqref{gen odd case intermediate} to complete the proof of \eqref{odd case}.

Next we prove \eqref{even case}.  The steps are similar to the proof of \eqref{odd case}, so we will be brief.  Assume that $k$ is even in equation \eqref{almost done 1}, and let $u\mapsto n u/r$.  Then we find
\begin{align*}
8&\left(1+24\cdot 2^{-k}+2^{11-2k}\right)L(\Delta,k)\\
&\qquad=\frac{(2\pi)^{k}}{(k-1)!}\int_{0}^{\infty}u^{k-7}\sum_{\substack{n,s\ge 1\\\text{$s$ odd}}}n^{k-1} e^{-2\pi n s u}\sum_{\substack{m,r\ge 1\\\text{$m$,$r$ odd}}}\frac{1}{r^{k-11}}e^{-\frac{2\pi m r}{4u}}\d u\\
&\qquad=\frac{(2\pi)^{k}}{(k-1)!}\int_{0}^{\infty}4(4u)^{5-k}\sum_{\substack{n,s\ge 1\\\text{$s$ odd}}}n^{k-1} e^{-\frac{2\pi n s}{4 u}}\sum_{\substack{m,r\ge 1\\\text{$m$,$r$ odd}}}\frac{1}{r^{k-11}}e^{-2\pi m r u}\d u,\\
&\qquad=\frac{(2\pi)^{2k-11}}{(k-1)!(k-12)!}\int_{0}^{\infty}\int_{u}^{\infty}(z-u)^{k-12}\left( 4(4u)^{5-k}\sum_{\substack{n,s\ge 1\\\text{$s$ odd}}}n^{k-1} e^{-\frac{2\pi n s}{4 u}}\right)\\
&\qquad\qquad\qquad\qquad\qquad\qquad\qquad\times\left(\sum_{\substack{m,r\ge 1\\\text{$m$,$r$ odd}}}m^{k-11}e^{-2\pi m r z}\right)\d z ~\d u.
\end{align*}
The second equality follows from mapping $u\mapsto\frac{1}{4u}$, and the third equality follows from \eqref{shifted gamma integral}.  Finally, it is easy to show that
\begin{align*}
4(4u)^{5-k}\sum_{\substack{n,s\ge 1\\\text{$s$ odd}}}n^{k-1} e^{-\frac{2\pi n s}{4 u}}=&-i^{k} 2^{11-k}\zeta(1-k) u^5 R_{k}(u),\\
\sum_{\substack{m,r\ge 1\\\text{$m$,$r$ odd}}}m^{k-11}e^{-2\pi m r z}=&\frac{1}{2}\zeta(11-k)S_{k-10}(z),
\end{align*}
where $R_{k}(u)$ and $S_{k}(z)$ are defined in \eqref{R def} and \eqref{S def}.
\end{proof}

\begin{corollary}\label{Values at k bigger than 12} Let $F(\alpha)$ denote the usual hypergeometric function, defined in \eqref{2F1 def}.  The following identities are true:
\begin{align}
L(\Delta,13)
=&\frac{\pi^{13}}{q_{13}}\int_{0}^{1}\int_{0}^{\alpha}\left[F(\alpha)F(1-\beta)
-F(\beta)F(1-\alpha)\right]\left[F(\alpha)F(\beta)\right]^5\notag\\
&\qquad\times\frac{(1+\alpha)(17-32\alpha+17\alpha^2)(2+13\beta+2\beta^2)}{\alpha(1-\beta)}\d \beta~ \d \alpha,\label{L value at 13}\\
L(\Delta,14)=&\frac{\pi^{15} }{q_{14}}\int_0^1\int_0^{\alpha} \left[F(\alpha) F(1-\beta)- F(\beta)F(1-\alpha)\right]^2 \left[F(\alpha)F(1-\alpha)\right]^5\notag\\
&\qquad\times\frac{(2-\alpha) \left(5461-10922 \alpha+5973 \alpha^2-512 \alpha^3+\alpha^4\right) (2-\beta)}{\alpha (1-\beta)}\d\beta~ \d\alpha,\label{L value at 14}\\
L(\Delta,15)=&\frac{\pi ^{17} }{q_{15}}\int _0^1\int _0^{\alpha}  [F(\alpha) F(1-\beta)-F(\beta) F(1-\alpha) ]^3 [F(\alpha) F(\beta)]^5\notag\\
&\qquad\times\frac{\left(31-47 \alpha+33 \alpha^2-47 \alpha^3+31 \alpha^4\right) (1+\beta) \left(1+29\beta+\beta^2\right) }{\alpha (1-\beta)}\d\beta~ \d\alpha\label{L value at 15},
\end{align}
where $q_{13}=122987403000$, $q_{14}=798232309875$, and $q_{15}=67002093132975/4$.
\end{corollary}
\begin{proof}  We sketch the proof of equations \eqref{L value at 13} and \eqref{L value at 15} below.  Set
\begin{align*}
u=\frac{F(1-\alpha)}{2F(\alpha)},&&z=\frac{F(1-\beta)}{2F(\beta)},
\end{align*}
and then apply \eqref{duda}. Equation \eqref{odd case} becomes
\begin{equation*}
\begin{split}
8&\left(1+24\cdot 2^{-k}+2^{11-2k}\right)L(\Delta,k)\\
&\qquad=\frac{2^{k-7} i^{k-1}\pi^{2k-11}\left[\zeta(6-k)\right]^2}{(k-1)!(k-12)!}\\
&\qquad\qquad\times\int_{0}^{1}\int_{0}^{\alpha}\left(\frac{F(1-\beta)}{2F(\beta)}-\frac{F(1-\alpha)}{2F(\alpha)}\right)^{k-12}\frac{P_{k-5}(u)Q_{k-5}(z)~\d \beta~\d \alpha}{4\pi^2\alpha(1-\alpha)\beta(1-\beta)F(\alpha)^2F(\beta)^2}.
\end{split}
\end{equation*}
Finally we use the formulas
\begin{align*}
P_{8}(u)=&15\left(1-\alpha^2\right)\left(17-32\alpha+17\alpha^2\right)\left[F(\alpha)\right]^8,\\
Q_{8}(z)=&15\beta\left(2+13\beta+2\beta^2\right)\left[F(\beta)\right]^8,\\
P_{10}(u)=&33 (1-\alpha) \left(31-47 \alpha+33 \alpha^2-47 \alpha^3+31 \alpha^4\right)\left[ F(\alpha)\right]^{10},\\
Q_{10}(z)=&-\frac{33}{2}\beta(1+\beta) \left(1+29 \beta+\beta^2\right)\left[ F(\beta)\right]^{10},
\end{align*}
when $k=13$ and $k=15$.  These formulas are easy to prove by expressing $E_{k}(i u)$ in terms of polynomials in $E_{4}(i u)$ and $E_{6}(i u)$, and then appealing to \cite[pg.~126, Entry~13]{Be3}.  In practice, we used numerical searches to find linear dependencies between $P_{k}(u)$, $Q_{k}(u)$, and $\left\{\alpha^{j} \left[F(\alpha)\right]^k\right\}_{j=0}^{j=k}$.
\end{proof}

If $u$ and $\alpha$ are related by \eqref{variable change}, then it is a classical fact that $E_{k}(2^{j} i u)=(\text{polynomial in $\alpha$})\times [F(\alpha)]^k$ for $j\in\{0,1,2\}$ and $k$ even.  This makes it possible to see that the pattern of Corollary \ref{Values at k bigger than 12} continues.  The identities for $L(\Delta,k)$ are always two dimensional integrals containing $$\left[F(\alpha)F(1-\beta)-F(\beta)F(1-\alpha)\right]^{k-12}.$$
Furthermore, if $k$ is even the integral contains $\left[F(\alpha)F(1-\alpha)\right]^5$, and if $k$ is odd the integral contains $\left[F(\alpha)F(\beta)\right]^5$.

\section{Speculation and Conclusion}\label{conclusion}

We have proved that $L(\Delta,k)$ is a period for $k\ge 12$. It is interesting to speculate on what simplifications might be possible for the integrals in Theorem \ref{Theorem value at 12} and Corollary \ref{Values at k bigger than 12}.  We can draw an analogy with the case of weight two modular forms.  If we select a specific modular form such as $f(z)=\eta(z)\eta(3z)\eta(5z)\eta(15z)$, then it is possible to prove results like
\begin{equation}\label{deninger}
\frac{15}{4\pi^2}L(f,2)=\int_{0}^{1}\int_{0}^{1}\log\left|1+X+X^{-1}+Y+Y^{-1}\right|\d t~\d s,
\end{equation}
where $X=e^{2\pi i t}$ and $Y=e^{2\pi i s}$.  The integral on the right is a Mahler measure, and the surface obtained from setting $1+X+X^{-1}+Y+Y^{-1}=0$, is precisely the elliptic curve attached to $f(z)$ by the modularity theorem.  Rodriguez-Villegas gave a very nice explanation of \textit{why} results like \eqref{deninger} exist \cite{RV}, \cite{De}, \cite{Bo1}, \cite{Rg2}.  The key point for us, is that the proof of \eqref{deninger} \textit{does not require any prior knowledge of the elliptic curve attached to $f(z)$}.  Thus it seems plausible that there might be an analogous relation between $L(\Delta,12)/\pi^{12}$ and the Mahler measure of a polynomial in $12$ variables.  If such an identity exists, then it should be possible to derive it from equation \eqref{intro formula} with only calculus.  Furthermore, the $12$ variable polynomial should give an affine model of a hypersurface attached to $\Delta(z)$.  Deligne found such a hypersurface in his proof of the Ramanujan-Petersson conjectures, but his hypersurface is typically described using \'{e}tale cohomology.

\begin{acknowledgements} The author thanks Wadim Zudilin for his encouragement and useful suggestions.
\end{acknowledgements}


\begin{thebibliography}{99}








\bibitem{BBGW}  \textsc{J.~M.~Borwein}, \textsc{D.~Borwein}, \textsc{M.~L.~Glasser}, and \textsc{J.~Wan},
Moments of Ramanujan's generalized elliptic integrals and extensions of Catalan's constant,
\emph{J. Math. Anal. Appl.} \textbf{384} (2011), no.~2, 478--496.

\bibitem{Be} A. Beilinson, Higher regulators of modular curves,
in: \emph{Applications of Algebraic K theory to Algebraic Geometry and Number Theory.}
Contemp. Math. {\bf 55}, Amer. Math. Soc., Providence (1986),1--34.

\bibitem{Be3}
\textsc{B.\,C.~Berndt}, \emph{Ramanujan's Notebooks, Part III}
(Springer-Verlag, New York, 1991).

\bibitem{Bo1}
\textsc{D.\,W.~Boyd},
Mahler's measure and special values of $L$-functions,
\emph{Experiment. Math.} \textbf{7} (1998), 37--82.


\bibitem{Bo} \textsc{D.~Boyd}, Mahler's measure and special values of $L$-functions of elliptic curves at $s=3$,
Preprint (2006). \url{http://www.math.ubc.ca/~boyd/sfu06.ed.pdf}

\bibitem{De}
\textsc{C.~Deninger},
Deligne periods of mixed motives, $K$-theory and the entropy of certain $\mathbb Z^n$-actions,
\emph{J. Amer. Math. Soc.} \textbf{10} (1997), no.~2, 259--281.

\bibitem{DS}
\textsc{C.~Deninger} and \textsc{A.~Scholl},
The Beilinson conjectures,
in: \emph{$L$-functions and Arithmetic,}
London Math. Soc. Lecture Notes {\bf{153}} (eds. J. Coates and M. J. Taylor), Cambridge Univ. Press, Cambridge (1991), 173--209.

\bibitem{KZ}
\textsc{M.~Kontsevich} and \textsc{D.~Zagier},
Periods,
in: \emph{Mathematics unlimited\,---\,2001 and beyond}
(Springer, Berlin, 2001), 771--808.

\bibitem{M}
\textsc{L.~J.~Mordell},
On Mr. Ramanujan's Empirical Expansions of Modular Functions,
\emph{Proc. Cambridge Phil. Soc.} {\bf{19}} (1917), 117--124.

\bibitem{Ono2} K.~Ono, The Web of Modularity: Arithmetic of the Coefficients of Modular Forms and q-series, {\em
 American Mathematical Society, Providence, RI,\/} 2004.

\bibitem{Ra}
\textsc{S.~Ramanujan}, On certain Arithmetical Functions, [Trans. Camb. Phil. Soc. \textbf{2} (1916), no.~9, 159--184].
\textit{Collected papers of Srinivasa Ramanujan, 23-29, AMS Chelsea
Publ., Providence, RI, 2000.}

\bibitem{Ran}
\textsc{R. A.~Rankin}
\textit{Modular Forms and Functions}, Cambridge University Press, Cambridge, 1977.


\bibitem{RV}
\textsc{F.~Rodriguez-Villegas},
Modular Mahler measures I,
in: \emph{Topics in number theory} (University Park, PA, 1997),
Math. Appl. \textbf{467} (Kluwer Acad. Publ., Dordrecht, 1999), 17--48.

\bibitem{RVTV} \textsc{F.~Rodriguez-Villegas}, \textsc{R.~Toledano} and \textsc{J.~D.~Vaaler},
Estimates for Mahler's measure of a linear form
\emph{Proc. Edinb. Math. Soc.} {\bf 47} (2004), 473–494.

\bibitem{RgRamanujan} \textsc{M.~D.~Rogers},
New $_5F_4$ hypergeometric transformations,
three-variable Mahler measures, and formulas for $1/\pi$,
\emph{Ramanujan J.} \textbf{18} (2009), no.~3, 327--340.


\bibitem{Rg1}
\textsc{M.~Rogers} and \textsc{W.~Zudilin},
From $L$-series of elliptic curves to Mahler measures,
\emph{Compositio Math.} \textbf{148} (2012), no.~2, 385-–414.

\bibitem{Rg2}
\textsc{M.~Rogers} and \textsc{W.~Zudilin},
On the Mahler measure of $1+X+1/X+Y+1/Y$,
to appear in \emph{Int. Math. Res. Notices}.


\bibitem{Sc}
\textsc{B.~Schoeneberg}
\emph{Elliptic modular functions: an introduction}. Springer-Verlag, New York, 1974.
Translated from the German by J. R. Smart and E. A. Schwandt, Die Grundlehren der mathematischen
Wissenschaften, Band 203.

\bibitem{W} \textsc{J.~G.~Wan},
Moments of products of elliptic integrals,
\emph{Adv. Appl. Math.} \textbf{48} (2012), no.~1, 121--141.

\bibitem{ZA}
\textsc{D.~Zagier}
Introduction to modular forms.
in: \emph{From Number Theory to Physics}, M. Waldschmidt et al, Springer-Verlag, Heidelberg (1992) 238--291.

\bibitem{Z}
\textsc{Y.~Zhou},
Legendre Functions, Spherical Rotations, and Multiple Elliptic Integrals,
preprint \texttt{arXiv:\,1301.1735 [math.CA]} (2013).

\bibitem{ZP}
\textsc{W.~Zudilin}, Period(d)ness of $L$-values,
\textit{Proceedings of the International Conference in memory of Alf van der Poorten (to appear)}, J.M. Borwein et al. (eds.), 14 pages.

\bibitem{Zt}
\textsc{W.~Zudilin},
``Hypergeometric evaluations of L-values of an elliptic curve", A plenary talk, Ramanujan-125 Conference "The Legacy of Srinivasa Ramanujan" (University of Delhi, New Delhi, India, December 17–22, 2012)
\url{http://wain.mi.ras.ru/PS/NewDelhi-slides.pdf}

\end{thebibliography}
\end{document}